\documentclass{amsart}

\usepackage{amsfonts}
\usepackage{amsmath}
 \usepackage{amssymb}
 \usepackage{setspace}
\usepackage{hyperref}

 \usepackage{epsfig}        
 \usepackage{epic,eepic}       

\newtheorem{theorem}{Theorem}
\theoremstyle{plain}

\newtheorem{corollary}{Corollary}

\newtheorem{lemma}{Lemma}

\newtheorem{remark}{Remark}

\numberwithin{equation}{section}

\begin{document}

\title[Approximating $n$-th differentiable functions of two variables]
{Approximating $n$-th differentiable functions of two variables
and Mid-point formula}

\author[M. W. Alomari]{Mohammad W. Alomari}
\address{ Department of Mathematics, Faculty of Science and
Information Technology, Irbid National University, 2600 Irbid
21110, Jordan.} \email{mwomath@gmail.com}

\date{\today}


\begin{abstract}
In this work,  approximations for real two variables function $f$
which has continuous partial $(n-1)$-derivatives $(n \ge 1)$ and
has the $n$--th partial derivative of bounded bivariation or
absolutely continuous are established.   Explicit bounds for this
representation are given. An approximation of a function $f$ by
its mid-point formula with its error is established.
\end{abstract}

\maketitle

\section{Introduction}

The Trapezoid formula
\begin{align}
\label{eq1.1}\int_a^b{f(x)dx}  \approx (b-a)\frac{f(a)+f(b)}{2}
\end{align}
is one of the most popular rule that approximate the definite
integral $\int_a^b{f(x)dx}$ of a real function $f$ over an
interval $[a,b]$.

In the last forty years, a great attention and valuable efforts
were devoted to investigate this formula in several ways and for
various type of functions. Among others, the generalized trapezoid
formula (GFT)
\begin{align}
\label{eq1.2} \frac{1}{{b - a}}\left[ {\left( {x - a}
\right)f\left( a \right) + \left( {b - x} \right)f\left( b
\right)} \right]
\end{align}
was expanded by Cerone et al. in \cite{Cerone}, and eligible to be
considered as one of the important works in the latest presented
literature. So that instead of \eqref{eq1.1} one may use
\eqref{eq1.2}, i.e.,
\begin{align}
\label{eq1}\int_a^b{f(t)dt}  \approx \frac{1}{{b - a}}\left[
{\left( {x - a} \right)f\left( a \right) + \left( {b - x}
\right)f\left( b \right)} \right],\qquad \forall x\in [a,b].
\end{align}
In another approach, if $f: [a,b] \to \mathbb{R}$ is assumed to be
bounded on $[a,b]$. The chord that connects its end points $A =
(a, f (a))$ and $B = (b, f (b))$ has the equation
\begin{align*}
d_f :\left[ {a,b} \right] \to \mathbb{R},\,\,\,\,\,\,d_f \left(
{x} \right) = \frac{1}{{b - a}}\left[ {f\left( a \right)\left( {b
- {x}} \right) + f\left( b \right)\left( {{x} - a} \right)}
\right].
\end{align*}
Some error approximations for the value of the function
$f(\bf{x})$ by the functional
\begin{align*}
\Phi _f \left( {x} \right): = \frac{{b - {x}}}{{b - a}} \cdot
f\left( a \right) + \frac{{{x} - a}}{{b - a}} \cdot f\left( b
\right) - f\left( {x} \right).
\end{align*}
were introduced by Dragomir in \cite{Dragomir1}, e.g., the
following bounds for $\Phi _f \left( {x} \right)$ are hold:
\begin{theorem}
\label{thm.DragomirSb1}If $f:[a,b] \to \mathbb{R}$ is of bounded
variation, then
\begin{align}
\label{eq.DragomirSb}\left| {\Phi _f \left( {x} \right)} \right|
&\le \left( {\frac{{b - {x}}}{{b - a}}} \right) \cdot \bigvee_a^x
\left( f \right) + \left( {\frac{{{x} - a}}{{b - a}}} \right)
\cdot \bigvee_x^b \left( f \right)
\nonumber\\
&\le\left\{ \begin{array}{l}
 \left[ {\frac{1}{2} + \left| {\frac{{{x} - {\textstyle{{a + b} \over 2}}}}{{b - a}}} \right|} \right] \cdot \bigvee_a^b \left( f \right) \\
 \left[ {\left( {\frac{{b - {x}}}{{b - a}}} \right)^p  + \left( {\frac{{{x} - a}}{{b - a}}} \right)^p } \right]^{{\textstyle{1 \over p}}}  \cdot \left[ {\left( {\bigvee_a^x \left( f \right)} \right)^q  + \left( {\bigvee_x^b \left( f \right)} \right)^q } \right]^{{\textstyle{1 \over q}}};  \\
   \,\,\,\,\, \,\,\,\,\, \,\,\,\,\, \,\,\,\,\, \,\,\,\,\, \,\,\,\,\, \,\,\,\,\, \,\,\,\,\, \,\,\,\,\, \,\,\,\,\, \,\,\,\,\, \,\,\,\,\, \,\,\,\,\, \,\,\,\,\, \,\,\,\,\, \,\,\,\,\,\,\,\,\,\,p>1, \frac{1}{p}+ \frac{1}{q}=1
  \\
 \frac{1}{2}\bigvee_a^b \left( f \right) + \frac{1}{2}\left| {\bigvee_a^x \left( f \right) - \bigvee_x^b \left( f \right)} \right| \\
 \end{array} \right.
\end{align}
The first inequality in (\ref{eq.DragomirSb}) is sharp. The
constant $\frac{1}{2}$ is best possible in the first and third
branches.
\end{theorem}
In 2008, Dragomir  \cite{Dragomir2}   provided an approximation
for the function $f$ which possesses continuous derivatives up to
the order $n - 1$ $(n \ge 1)$ and has the $n$-th derivative of
bounded variation, in terms of the chord that connects its end
points $A = (a, f (a))$ and $B = (b, f (b))$ and some more terms
which depend on the values of the $k$ derivatives of the function
taken at the end points $a$ and $b$, where $k$ is between $1$ and
$n$.

The following representation of $f$ was presented by Dragomir in
\cite{Dragomir2}, as follows:
\begin{theorem}
\label{thm.DragomirSb}Let $I$ be a closed subinterval on
$\mathbb{R}$, let $a, b \in I$ with $a < b$ and let $n$ be a
nonnegative integer. If $f : I \to \mathbb{R}$ is such that the
$n$-th derivative $f^{(n)}$ is of bounded variation on the
interval $[a, b]$, then for any $x \in [a,b]$ we have the
representation
\begin{multline}
f\left( x \right) = \frac{1}{{b - a}}\left[ {\left( {b - x}
\right)f\left( a \right) + \left( {x - a} \right)f\left( b
\right)} \right]
\\
+ \frac{{\left( {b - x} \right)\left( {x - a} \right)}}{{b - a}}
\cdot \sum\limits_{k = 1}^n {\frac{1}{{k!}}\left[ {\left( {x - a}
\right)^{k - 1} f^{\left( k \right)} \left( a \right) + \left( { -
1} \right)^k \left( {b - x} \right)^{k - 1} f^{\left( k \right)}
\left( b \right)} \right]}
\\
+ \frac{1}{{b - a}}\int_a^b {S_n \left( {x,t} \right)d\left(
{f^{\left( n \right)} \left( t \right)} \right)},
\end{multline}
where the kernel $S_n : [a,b]^2 \to \mathbb{R}$ is given by
\begin{align*}
S_n \left( {x,t} \right) = \left\{ \begin{array}{l}
 \left( {x - t} \right)^n \left( {b - x} \right),\,\,\,\,\,\,\,\,\,\,\,\,\,\,\,\,\,\,\,\,\,\,\,\,\,\,\,a \le t \le x \\
  \\
 \left( { - 1} \right)^{n + 1} \left( {t - x} \right)^n \left( {x - a} \right),\,\,\,\,\,a \le t \le x \\
 \end{array} \right.
\end{align*}
and the integral in the remainder is taken in the
Riemann--Stieltjes sense.
\end{theorem}
On utilizing the notations
\begin{align}
D_n \left( {f;x,a,b} \right)&:= \frac{1}{{b - a}}\left[ {\left( {b
- x} \right)f\left( a \right) + \left( {x - a} \right)f\left( b
\right)} \right]
\nonumber\\
&+ \frac{{\left( {b - x} \right)\left( {x - a} \right)}}{{b - a}}
\cdot \sum\limits_{k = 1}^n {\frac{1}{{k!}}\left[ {\left( {x - a}
\right)^{k - 1} f^{\left( k \right)} \left( a \right) + \left( { -
1} \right)^k \left( {b - x} \right)^{k - 1} f^{\left( k \right)}
\left( b \right)} \right]}
\end{align}
with
\begin{align}
E_n \left( {f;x,a,b} \right):= \frac{1}{{b - a}}\int_a^b {S_n
\left( {x,t} \right)d\left( {f^{\left( n \right)} \left( t
\right)} \right)},
\end{align}
and under the assumptions of Theorem \ref{thm.DragomirSb},
Dragomir provided an error approximation of the function $f$ using
the formula $f (x) - D_n \left( {f;x,a,b} \right) = E_n \left(
{f;x,a,b} \right)$ for any $x \in [a, b]$.

The concept of Riemann--Stieltjes  ($\mathcal{RS}$) double
integral $\int\limits_a^b {\int\limits_c^d {f\left( {x,y}
\right)d_x d_y \alpha\left( {x,y} \right)} }$  was characterized
by Fr\'{e}chet in \cite{Frechet}, and investigated later on by
Clarkson in \cite{Clarkson}. Such  $\mathcal{RS}$-integral  plays
an important role in Mathematics with multiple applications in
several subfields including Probability Theory \& Statistics,
Complex Analysis, Functional Analysis, Operator Theory and others.

For $a,b,c,d \in \mathbb{R}$, we consider the subset $Q:= \left\{
{\left( {x,y} \right):a \le x \le b,c \le y \le d} \right\}$ of
$\mathbb{R}^2$. If $P : = \left\{ {\left( {x_i ,y_j } \right):x_{i
- 1}  \le x \le x_i \,;\,y_{j - 1}  \le y \le y_j \,;\,i = 1,
\ldots ,n\,;\,j = 1, \ldots ,m} \right\}$ is a partition of $Q$,
write
\begin{align*}
\Delta _{11} f\left( {x_i ,y_j } \right) = f\left( {x_{i - 1}
,y_{j - 1} } \right) - f\left( {x_{i - 1} ,y_j } \right) - f\left(
{x_i ,y_{j - 1} } \right) + f\left( {x_i ,y_j } \right)
\end{align*}
for $i = 1,2, \cdots, n$ and $j = 1,2, \cdots, m$. The function
$f(x,y)$ is said to be of  bounded variation in the Vitali sense
(or simply bounded bivariation \cite{Ghorpade}) if there exists a
positive quantity $M$ such that for every partition on $Q$ we have
$\sum\limits_{i = 1}^n {\sum\limits_{j = 1}^m {\left| {\Delta
_{11} f\left( {x_i ,y_j } \right)} \right|} } \le M$ (see
\cite{ClarksonAdams}). Consequently,  the  total bivariation of a
function of two variables,  over bidimensional interval $Q$, is
defined to be the number
\begin{align*}
\bigvee_Q \left( {f} \right):= \bigvee_c^d \bigvee_a^b \left( {f}
\right):= \sup \left\{ {\sum P :P \in \mathcal{P}\left( Q\right)}
\right\},
\end{align*}
where $\sum {(P)}$ denote the sum $\sum\limits_{i = 1}^n
{\sum\limits_{j = 1}^m {\left| {\Delta _{11} f\left( {x_i ,y_j }
\right)} \right|} }$ corresponding to the partition $P$ of $Q$.
For further properties of mappings of bounded bivariation we refer
the interested reader to  the book recent \cite{Ghorpade} and the
monograph \cite{alomari}.

In 2009, Alomari \cite{alomari} has studied the approximation
problem of the Riemann--Stieltjes double integrals $\int_c^d
{\int_a^b {f\left( {t,s} \right)d_t d_s \alpha \left( {t,s}
\right)} }$ for various kind of the integrand $f$ and the
integrator $\alpha$ via inequalities approach. Among others, we
need the following two basic results from \cite{alomari}:
\begin{lemma}$($Integration by parts  \cite{alomari}$)$\label{lma1}
If $f \in \mathcal{RS}(\alpha)$ on $Q$, then $\alpha \in
\mathcal{RS}(f)$ on $Q$, and we have
\begin{multline}
\int_c^d {\int_a^b {f\left( {t,s} \right)d_t d_s \alpha \left(
{t,s} \right)} }  + \int_c^d {\int_a^b {\alpha \left( {t,s}
\right)d_t d_s f\left( {t,s} \right)} }
\\
= f\left( {b,d} \right)\alpha \left( {b,d} \right) - f\left( {b,c}
\right)\alpha \left( {b,c} \right) - f\left( {a,d} \right)\alpha
\left( {a,d} \right) + f\left( {a,c} \right)\alpha \left( {a,c}
\right).
\end{multline}
\end{lemma}

\begin{lemma}\label{lma3}$($\cite{alomari}$)$
Assume that $g \in \mathcal{RS}{\left( \alpha \right)}$ on $Q$ and
$\alpha$ is of bounded bivariation on $Q$, then the inequality
\begin{align}
\left| {\int_c^d {\int_a^b {g\left( {x,y} \right)d_x d_y \alpha
\left( {x,y} \right)} } } \right| \le \mathop {\sup
}\limits_{\left( {x,y} \right) \in Q} \left| {g\left( {x,y}
\right)} \right| \cdot  \bigvee_{Q} \left( \alpha \right),
\end{align}
holds.
\end{lemma}

A small attention and a few works have been considered for
mappings of two variables; particularly, the approximation problem
of $\mathcal{RS}$-double integral $\int_a^b {\int_c^d {f\left(
{t,s} \right)d_s d_t \alpha\left( {t,s} \right)} }$ in terms of
$\mathcal{RS}$-double sums. Although the important applications in
approximations, numerical integration in two or more variables is
still much less developed area than its one-dimensional
counterpart, which has been worked on intensively. For recent
inequalities involving functions of two variables the reader may
refer to \cite{alomari}--\cite{BarnettDragomir1},
\cite{DragomirCeroneBarnettRoumeliotis}--\cite{DragomirRassias}
and \cite{Hanna}--\cite{Zhongxue}.

In this paper, we extend Cerone et al \cite{Cerone} and Dragomir
\cite{Dragomir1}, \cite{Dragomir2} results (mainly Theorem
\ref{thm.DragomirSb}) by giving bidimensional representation for
functions whose $n$-th partial derivatives are of bounded
bivariation or absolutely continuous. Explicit bounds for this
representation are given. An approximation of a function $f$ by
its mid-point formula with its error is established.

\section{The Result}

For a bounded function $f:Q\to \mathbb{R}$ the chord plane that
connects its end points may be represented as:
\begin{align*}
C_f\left({x,y}\right):=\frac{1}{(b - a)(d - c)}\left[ {(b - x)(d -
y)f(a,c) + (b - x)(y - c) f(a,d)} \right.
\\
\left. {+(x - a)(d - y) f(b,c) + (x - a)(y - c)f(b,d) } \right],
\end{align*}
which is the generalized Trapezoid formula for functions of two
variables. Consequently, the difference between the function $f$
and its chord plane may considered by the functional:
\begin{align*}
\psi\left({f,C}\right)=C_f\left({x,y}\right)-f\left({x,y}\right).
\end{align*}
One may study this difference for various type of functions. On
the other hand,  this representation can be extended to $n$-th
differentiable functions, as follows:
\begin{theorem}
\label{thm5.6.1} Let $f:Q\to \mathbb{R}$ be a real valued function
which has continuous $(n-1)$ partial derivatives ($n\ge 1$). If
the $n$-th partial derivatives $D^{n}f$ ($n\ge1$) are of bounded
bivariation on $Q$, then  for any $(x,y)\in Q$ we have the
representation
\begin{multline}
\label{eq5.6.1}f\left( {x,y} \right) = \frac{1}{(b - a)(d -
c)}\left[ {(b - x)(d - y)f(a,c) + (b - x)(y - c) f(a,d)} \right.
\\
\left. {+(x - a)(d - y) f(b,c) + (x - a)(y - c)f(b,d) } \right] +
\frac{(y-c)(d - y)}{(b - a)(d - c)}
\\
\times\sum\limits_{j = 1}^n \frac{1}{{j!}}\left(
{\begin{array}{*{20}c}
   n  \\
   j  \\
\end{array}} \right)
\left\{ {(b - x)\left( {x - a } \right)^{n - j} \left[ { \left( {y
- c } \right)^{j-1} D^{n}f\left( {a,c } \right)} \right.} \right.
+ \left. {(-1)^j\left( {d - y } \right)^{j-1} D^{n}f\left( {a,d }
\right)} \right]
\\
+ \left. {(x - a)\left( {b - x } \right)^{n - j} \left[ {
(-1)^j\left( {y - c } \right)^{j-1} D^{n}f\left( {b,c } \right)+
\left( {d - y } \right)^{j-1} D^{n}f\left( {b,d } \right)} \right]
} \right\}
\\
+ \frac{1}{{\left( {b - a} \right)\left( {d - c} \right)}}\int_c^d
{\int_a^b {S_n \left( {x,t;y,s} \right)d_t d_s \left( {D^n f\left(
{t,s} \right)} \right)}}
\end{multline}
where, $D^{n}f\left( {t,s } \right) = \frac{{\partial ^n
f}}{{\partial t^{n - j} \partial s^j }}\left( {t,s } \right)$
$(j=0,1,\cdots,n)$ and
\begin{align*}
S_n \left( {x,t;y,s} \right) = \frac{1}{{n!}}\left\{
\begin{array}{l}
 \left( {x - t} \right)^n \left( {b - x} \right)\left( {y - s} \right)^n \left( {d - y} \right),\,\,\,\,\,\,\,\,\,\,\,\,\,\,\,\,\,\,\,a \le t \le x,\,\,\,\,\,c \le s \le y \\
 \left( { - 1} \right)^n \left( {t - x} \right)^n \left( {x - a} \right)\left( {y - s} \right)^n \left( {d - y} \right),\,\,\,x < t \le b,\,\,\,\,\,c \le s \le y \\
 \left( { - 1} \right)^n \left( {x - t} \right)^n \left( {b - x} \right)\left( {s - y} \right)^n \left( {y - c} \right),\,\,\,\,\,a \le t \le x,\,\,\,\,\,y < s \le d \\
 \left( {t - x} \right)^n \left( {x - a} \right)\left( {s - y} \right)^n \left( {y - c} \right),\,\,\,\,\,\,\,\,\,\,\,\,\,\,\,\,\,\,\,\,\,x < t \le b,\,\,\,\,\,y < s \le d \\
 \end{array} \right.
\end{align*}
\end{theorem}

\begin{proof}
We utilize the following Taylor's representation formula for
functions $f : Q \subset \mathbb{R}^2\to \mathbb{R}$ such that the
$n$-th partial derivatives $D^{n}f$ are of locally bounded
bivariation on $Q$,
\begin{align}
\label{eq5.6.2}f\left( {x,y} \right) = P_n \left( {x,y} \right) +
R_n \left( {x,y} \right)
\end{align}
such that,
\begin{align}
\label{eq5.6.2}P_n \left( {x,y} \right) = \sum\limits_{j = 0}^n {
\frac{1}{{j!}}\left( {\begin{array}{*{20}c}
   n  \\
   j  \\
\end{array}} \right)
\left( {x - x_0 } \right)^{n - j} \left( {y - y_0 } \right)^j
D^{n}f\left( {x_0 ,y_0 } \right)}
\end{align}
and
\begin{align}
\label{eq5.6.3}R_n \left( {x,y} \right) =
\frac{1}{{n!}}\int_{y_0}^y {\int_{x_0}^x {\left( {x - t} \right)^n
\left( {y - s} \right)^n d_t d_s \left( { D^{n}f\left( {t,s }
\right)} \right)}}
\end{align}
where, $D^{n}f\left( {t,s } \right) = \frac{{\partial ^n
f}}{{\partial t^{n - j} \partial s^j }}\left( {t,s } \right)$, and
$(x,y)$, $(x_0, y_0)$ are in $Q$ and the double integral in the
remainder is taken in the Riemann-Stieltjes sense.

Choosing $x_0 = a$, $y_0 = c$ and then $x_0 = b$, $y_0 = d$ in
(\ref{eq5.6.1}) we can write that
\begin{multline}
\label{eq5.6.4}f\left( {x,y} \right) = \sum\limits_{j = 0}^n {
\frac{1}{{j!}}\left( {\begin{array}{*{20}c}
   n  \\
   j  \\
\end{array}} \right)
\left( {x - a } \right)^{n - j} \left( {y - c } \right)^j
D^{n}f\left( {a,c } \right)}
\\
+ \frac{1}{{n!}}\int_{c}^y {\int_{a}^x {\left( {x - t} \right)^n
\left( {y - s} \right)^n d_t d_s \left( { D^{n}f\left( {t,s }
\right)} \right)}},
\end{multline}
\begin{multline}
\label{eq5.6.5}f\left( {x,y} \right) = \sum\limits_{j = 0}^n {
\frac{(-1)^j}{{j!}}\left( {\begin{array}{*{20}c}
   n  \\
   j  \\
\end{array}} \right)
\left( {x - a } \right)^{n - j} \left( {d - y } \right)^j
D^{n}f\left( {a,d } \right)}
\\
+ \frac{(-1)^{n+1}}{{n!}}\int_{y}^d {\int_{a}^x {\left( {x - t}
\right)^n \left( {s - y} \right)^n d_t d_s \left( { D^{n}f\left(
{t,s } \right)} \right)}},
\end{multline}
\begin{multline}
\label{eq5.6.6}f\left( {x,y} \right) = \sum\limits_{j = 0}^n {
\frac{(-1)^j}{{j!}}\left( {\begin{array}{*{20}c}
   n  \\
   j  \\
\end{array}} \right)
\left( {b - x} \right)^{n - j} \left( {y - c } \right)^j
D^{n}f\left( {b,c } \right)}
\\
+ \frac{(-1)^{n+1}}{{n!}}\int_{c}^y {\int_{x}^b {\left( {t - x}
\right)^n \left( {y - s} \right)^n d_t d_s \left( { D^{n}f\left(
{t,s } \right)} \right)}}
\end{multline}
\begin{multline}
\label{eq5.6.7}f\left( {x,y} \right) = \sum\limits_{j = 0}^n {
\frac{1}{{j!}}\left( {\begin{array}{*{20}c}
   n  \\
   j  \\
\end{array}} \right)
\left( {b - x} \right)^{n - j} \left( {d - y} \right)^j
D^{n}f\left( {b,d } \right)}
\\
+ \frac{1}{{n!}}\int_{y}^d {\int_{x}^b {\left( {t - x} \right)^n
\left( {s - y} \right)^n d_t d_s \left( { D^{n}f\left( {t,s }
\right)} \right)}}
\end{multline}
for ny $(x,y) \in Q$.

Now, by multiplying (\ref{eq5.6.4}) with $(b - x)(d - y)$,
(\ref{eq5.6.5}) with $(b - x)(y - c)$, (\ref{eq5.6.6}) with $(x -
a)(d - y)$, (\ref{eq5.6.7}) with $(x - a)(y - c)$, we get
\begin{multline}
\label{eq5.6.8}(b - x)(d - y)f\left( {x,y} \right)
\\
= (b - x)(d - y)f(a,c)+(b - x)(d - y)\sum\limits_{j = 1}^n {
\frac{1}{{j!}}\left( {\begin{array}{*{20}c}
   n  \\
   j  \\
\end{array}} \right)
\left( {x - a } \right)^{n - j} \left( {y - c } \right)^j
D^{n}f\left( {a,c } \right)}
\\
+ \frac{1}{{n!}}(b - x)(d - y)\int_{c}^y {\int_{a}^x {\left( {x -
t} \right)^n \left( {y - s} \right)^n d_t d_s \left( {
D^{n}f\left( {t,s } \right)} \right)}},
\end{multline}
\begin{multline}
\label{eq5.6.9}(b - x)(y - c)f\left( {x,y} \right)
\\
= (b - x)(y - c) f(a,d)+(b - x)(y - c)\sum\limits_{j = 0}^n {
\frac{(-1)^j}{{j!}}\left( {\begin{array}{*{20}c}
   n  \\
   j  \\
\end{array}} \right)
\left( {x - a } \right)^{n - j} \left( {d - y } \right)^j
D^{n}f\left( {a,d } \right)}
\\
+ \frac{(-1)^{n+1}}{{n!}}(b - x)(y - c)\int_{y}^d {\int_{a}^x
{\left( {x - t} \right)^n \left( {s - y} \right)^n d_t d_s \left(
{ D^{n}f\left( {t,s } \right)} \right)}},
\end{multline}
\begin{multline}
\label{eq5.6.10}(x - a)(d - y)f\left( {x,y} \right)
\\
 = (x - a)(d -
y) f(b,c)+(x - a)(d - y)\sum\limits_{j = 0}^n {
\frac{(-1)^j}{{j!}}\left( {\begin{array}{*{20}c}
   n  \\
   j  \\
\end{array}} \right)
\left( {b - x} \right)^{n - j} \left( {y - c } \right)^j
D^{n}f\left( {b,c } \right)}
\\
+ \frac{(-1)^{n+1}}{{n!}}(x - a)(d - y)\int_{c}^y {\int_{x}^b
{\left( {t - x} \right)^n \left( {y - s} \right)^n d_t d_s \left(
{ D^{n}f\left( {t,s } \right)} \right)}}
\end{multline}
\begin{multline}
\label{eq5.6.11}(x - a)(y - c)f\left( {x,y} \right)
\\
= (x - a)(y - c)f(b,d)+(x - a)(y - c)\sum\limits_{j = 0}^n {
\frac{1}{{j!}}\left( {\begin{array}{*{20}c}
   n  \\
   j  \\
\end{array}} \right)
\left( {b - x} \right)^{n - j} \left( {d - y} \right)^j
D^{n}f\left( {b,d } \right)}
\\
+ \frac{1}{{n!}}(x - a)(y - c)\int_{y}^d {\int_{x}^b {\left( {t -
x} \right)^n \left( {s - y} \right)^n d_t d_s \left( {
D^{n}f\left( {t,s } \right)} \right)}}
\end{multline}
respectively, for ny $(x,y) \in Q$.

Finally, by adding the equalities
(\ref{eq5.6.8})--(\ref{eq5.6.11}) and dividing the sum with $
\left( {b - a} \right)\left( {d - c} \right)$, we obtain
\begin{align*}
f\left( {x,y} \right) &= \frac{1}{(b - a)(d - c)}\left[ {(b - x)(d
- y)f(a,c) + (b - x)(y - c) f(a,d)} \right.
\\
&\qquad\left. {+(x - a)(d - y) f(b,c) + (x - a)(y - c)f(b,d) }
\right]
\\
&\qquad+\frac{(y-c)(d - y)}{(b - a)(d - c)}\sum\limits_{j = 1}^n
\frac{1}{{j!}}\left( {\begin{array}{*{20}c}
   n  \\
   j  \\
\end{array}} \right)
\left\{ {(b - x)\left( {x - a } \right)^{n - j} \left[ { \left( {y
- c } \right)^{j-1} D^{n}f\left( {a,c } \right)} \right.} \right.
\\
&\qquad+ \left. {(-1)^j\left( {d - y } \right)^{j-1} D^{n}f\left(
{a,d } \right)} \right]+ (x - a)\left( {b - x } \right)^{n - j}
\left[ { (-1)^j\left( {y - c } \right)^{j-1} D^{n}f\left( {b,c }
\right)} \right.
\\
&\qquad\qquad\qquad\qquad\qquad\qquad\left. {+ \left. {\left( {d -
y } \right)^{j-1} D^{n}f\left( {b,d } \right)} \right] } \right\}
\\
&\qquad+ \frac{1}{{n!}} \frac{(b - x)(d - y)}{{\left( {b - a}
\right)\left( {d - c} \right)}} \int_{c}^y {\int_{a}^x {\left( {x
- t} \right)^n \left( {y - s} \right)^n d_t d_s \left( {
D^{n}f\left( {t,s } \right)} \right)}}
\\
&\qquad+ \frac{(-1)^{n+1}}{{n!}}\frac{(b - x)(y - c)}{{\left( {b -
a} \right)\left( {d - c} \right)}}\int_{y}^d {\int_{a}^x {\left(
{x - t} \right)^n \left( {s - y} \right)^n d_t d_s \left( {
D^{n}f\left( {t,s } \right)} \right)}},
\\
&\qquad+ \frac{(-1)^{n+1}}{{n!}}\frac{(x - a)(d - y)}{{\left( {b -
a} \right)\left( {d - c} \right)}}\int_{c}^y {\int_{x}^b {\left(
{t - x} \right)^n \left( {y - s} \right)^n d_t d_s \left( {
D^{n}f\left( {t,s } \right)} \right)}}
\\
&\qquad+ \frac{1}{{n!}}\frac{(x - a)(y - c)}{{\left( {b - a}
\right)\left( {d - c} \right)}}\int_{y}^d {\int_{x}^b {\left( {t -
x} \right)^n \left( {s - y} \right)^n d_t d_s \left( {
D^{n}f\left( {t,s } \right)} \right)}}
\end{align*}
which gives the desired representation (\ref{eq5.6.1}).
\end{proof}

\begin{remark}
The case $n=0$ provides the representation
\begin{multline}
\label{eq5.6.13}f\left( {x,y} \right) = \frac{1}{(b - a)(d -
c)}\left[ {(b - x)(d - y)f(a,c) + (b - x)(y - c) f(a,d)} \right.
\\
\left. {+(x - a)(d - y) f(b,c) + (x - a)(y - c)f(b,d) } \right]
\\
+ \frac{1}{{\left( {b - a} \right)\left( {d - c} \right)}}\int_c^d
{\int_a^b {S_0 \left( {x,t;y,s} \right)d_t d_s \left( {f\left(
{t,s} \right)} \right)}}
\end{multline}
for any $x \in Q$, where
\begin{align*}
S_0 \left( {x,t;y,s} \right) = \left\{ \begin{array}{l}
 \left( {b - x} \right)\left( {d - y} \right),\,\,\,\,\,\,\,a \le t \le x,\,\,\,\,\,c \le s \le y \\
 \left( {x - a} \right)\left( {d - y} \right),\,\,\,\,\,\,\,x < t \le b,\,\,\,\,\,c \le s \le y \\
 \left( {b - x} \right)\left( {y - c} \right),\,\,\,\,\,\,\,\,\,a \le t \le x,\,\,\,\,y < s \le d \\
 \left( {x - a} \right)\left( {y - c} \right),\,\,\,\,\,\,\,\,\,x < t \le b,\,\,\,\,y < s \le d \\
 \end{array} \right.
\end{align*}
and $f$ is of bounded variation on $Q$.
\end{remark}
The case when $n=0$,  provides an integral formula which compare
any value of $f(x,y)$, $(x,y)\in Q$ with the values of the
function and its derivatives at the rectangle end points (the
corners of the rectangle generated by the end points). More
specific, the Rectangular Mid-point value of $f$ can be
represented in the following corollary:
\begin{corollary}
With the assumptions of Theorem \ref{thm5.6.1} for $f$ and $Q$, we
have the identity
\begin{multline}
\label{eq5.6.14}f\left( {\frac{a+b}{2},\frac{c+d}{2}} \right) =
\frac{f(a,c)+ f(a,d) +f(b,c)+ f(b,d)}{4}
\\
+ \frac{1}{2^{n+2}}\sum\limits_{j = 1}^n \frac{1}{{j!}}\left(
{\begin{array}{*{20}c}
   n  \\
   j  \\
\end{array}} \right)
\left( {b - a } \right)^{n - j} \left( {d - c} \right)^{ j}\left\{
{ D^{n}f\left( {a,c } \right) + (-1)^j D^{n}f\left( {a,d }
\right)} \right.
\\
+ \left. { (-1)^j D^{n}f\left( {b,c } \right)+ D^{n}f\left( {b,d }
\right) } \right\}
\\
+ \frac{1}{{\left( {b - a} \right)\left( {d - c} \right)}}\int_c^d
{\int_a^b {M_n \left( {t,s} \right)d_t d_s \left( {D^n f\left(
{t,s} \right)} \right)}}
\end{multline}
where, $D^{n}f\left( {t,s } \right) = \frac{{\partial ^n
f}}{{\partial t^{n - j} \partial s^j }}\left( {t,s } \right)$ and
\begin{align*}
M_n \left( {t,s} \right) = \frac{\left( {b - a} \right)\left( {d -
c} \right)}{{4n!}}\left\{
\begin{array}{l}
 \left( {\frac{a+b}{2} - t} \right)^n \left( {\frac{c+d}{2} - s} \right)^n ,\,\,\,\,\,\,\,\,\,\,\,\,\,\,\,\,\,\,\,a \le t \le \frac{a+b}{2},\,\,\,\,\,c \le s \le \frac{c+d}{2} \\
 \left( { - 1} \right)^n \left( {t - \frac{a+b}{2}} \right)^n \left( {\frac{c+d}{2} - s} \right)^n,\,\,\,\frac{a+b}{2} < t \le b,\,\,\,\,\,c \le s \le \frac{c+d}{2} \\
 \left( { - 1} \right)^n \left( {\frac{a+b}{2} - t} \right)^n \left( {s - \frac{c+d}{2}} \right)^n,\,\,\,\,\,a \le t \le \frac{a+b}{2},\,\,\,\,\,\frac{c+d}{2} < s \le d \\
 \left( {t - \frac{a+b}{2}} \right)^n\left( {s - \frac{c+d}{2}} \right)^n,\,\,\,\,\,\,\,\,\,\,\,\,\,\,\,\,\,\,\,\,\,\frac{a+b}{2} < t \le b,\,\,\,\,\,\frac{c+d}{2} < s \le d \\
 \end{array} \right.
\end{align*}
\end{corollary}
On utilizing the following notations
\begin{multline}
\label{eq5.6.15}\mathcal{A}\left( {f,Q} \right) = \frac{1}{(b -
a)(d - c)}\left[ {(b - x)(d - y)f(a,c) + (b - x)(y - c) f(a,d)}
\right.
\\
\left. {+(x - a)(d - y) f(b,c) + (x - a)(y - c)f(b,d) } \right] +
\frac{(y-c)(d - y)}{(b - a)(d - c)}
\\
\times\sum\limits_{j = 1}^n \frac{1}{{j!}}\left(
{\begin{array}{*{20}c}
   n  \\
   j  \\
\end{array}} \right)
\left\{ {(b - x)\left( {x - a } \right)^{n - j} \left[ { \left( {y
- c } \right)^{j-1} D^{n}f\left( {a,c } \right)} \right.} \right.
+ \left. {(-1)^j\left( {d - y } \right)^{j-1} D^{n}f\left( {a,d }
\right)} \right]
\\
+ \left. {(x - a)\left( {b - x } \right)^{n - j} \left[ {
(-1)^j\left( {y - c } \right)^{j-1} D^{n}f\left( {b,c } \right)+
\left( {d - y } \right)^{j-1} D^{n}f\left( {b,d } \right)} \right]
} \right\}
\end{multline}
and
\begin{align}
\label{eq5.6.16}\mathcal{B}_n\left( {f,Q}
\right):=\frac{1}{{\left( {b - a} \right)\left( {d - c}
\right)}}\int_c^d {\int_a^b {S_n \left( {x,t;y,s} \right)d_t d_s
\left( {D^n f\left( {t,s} \right)} \right)}}
\end{align}
under the assumptions of Theorem \ref{thm5.6.1}, we can
approximate the function $f$ utilizing the polynomials
$\mathcal{A}_n\left( {f,Q} \right)$ with the error
$\mathcal{B}_n\left( {f,Q} \right)$ . In other words, we have
\begin{align*}
f(x,y) = \mathcal{A}_n\left( {f,Q} \right) + \mathcal{B}_n\left(
{f,Q} \right)
\end{align*}
for any $(x,y) \in Q$ .

 The error $\mathcal{B}_n\left(
{f,Q} \right)$ of a bounded bivariation function $f(x,y)$
satisfies the following bounds:
\begin{theorem}
\label{thm5.6.4}With the assumptions of Theorem \ref{thm5.6.1} for
$f$ and $Q$, we have
\begin{align}
\label{eq5.6.17} &\left| {\mathcal{B}_n\left( {f,Q} \right)}
\right|
\\
&\le   \frac{1}{{n!\left( {b - a} \right)\left( {d - c} \right)}}
 \cdot \max \left\{
{\left( {y - c} \right)^n \left( {d - y} \right),\left( {y - c}
\right)\left( {d - y} \right)^n } \right\}
\nonumber\\
&\qquad \times\max \left\{ {\left( {x - a} \right)^n \left( {b -
x} \right),\left( {x - a} \right)\left( {b - x} \right)^n }
\right\}
 \cdot
\bigvee_c^d\bigvee_a^b\left( {D^n f} \right)
\nonumber\\
&\le    \frac{{\left( {b - a} \right)^{n }\left( {d - c}
\right)^{n } }}{{2^{2n + 2} (n)! }}  \cdot
\bigvee_c^d\bigvee_a^b\left( {D^n f} \right)\nonumber
\end{align}
\end{theorem}

\begin{proof}
Using the inequality for the Riemann--Stieltjes integral of
continuous integrands and bounded bivariation integrators, we have
\begin{align*}
\left| {\mathcal{B}_n\left( {f,Q} \right)} \right| &= \left|
{\frac{1}{{n!\left( {b - a} \right)\left( {d - c} \right)}}\left[
{\int_c^y {\int_a^x {\left( {x - t} \right)^n \left( {b - x}
\right)\left( {y - s} \right)^n \left( {d - y} \right)} d_t d_s
\left( {D^n f\left( {t,s} \right)} \right)} } \right.} \right.
\\
&\qquad+ \int_y^d {\int_a^x {\left( { - 1} \right)^{n + 1} \left(
{x - t} \right)^n \left( {b - x} \right)\left( {s - y} \right)^n
\left( {y - c} \right)} d_t d_s \left( {D^n f\left( {t,s} \right)}
\right)}
\\
&\qquad+ \int_c^y {\int_x^b {\left( { - 1} \right)^{n + 1} \left(
{t - x} \right)^n \left( {x - a} \right)\left( {y - s} \right)^n
\left( {d - y} \right)}d_t d_s \left( {D^n f\left( {t,s} \right)}
\right)}
\\
&\qquad\left. {\left. { + \int_y^d {\int_x^b {\left( {t - x}
\right)^n \left( {x - a} \right)\left( {y - s} \right)^n \left( {d
- y} \right)} d_t d_s \left( {D^n f\left( {t,s} \right)} \right)}
} \right]} \right|
\\
&\le \frac{1}{{n!\left( {b - a} \right)\left( {d - c}
\right)}}\left[ {\left| {\int_c^y {\int_a^x {\left( {x - t}
\right)^n \left( {b - x} \right)\left( {y - s} \right)^n \left( {d
- y} \right)}d_t d_s \left( {D^n f\left( {t,s} \right)} \right)} }
\right|} \right.
\\
&\qquad+ \left| {\int_y^d {\int_a^x {\left( { - 1} \right)^{n + 1}
\left( {x - t} \right)^n \left( {b - x} \right)\left( {s - y}
\right)^n \left( {y - c} \right)} d_t d_s \left( {D^n f\left(
{t,s} \right)} \right)} } \right|
\\
&\qquad+ \left| {\int_c^y {\int_x^b {\left( { - 1} \right)^{n + 1}
\left( {t - x} \right)^n \left( {x - a} \right)\left( {y - s}
\right)^n \left( {d - y} \right)} d_t d_s \left( {D^n f\left(
{t,s} \right)} \right)} } \right|
\\
&\qquad\left. { + \left| {\int_y^d {\int_x^b {\left( {t - x}
\right)^n \left( {x - a} \right)\left( {y - s} \right)^n \left( {d
- y} \right)}d_t d_s \left( {D^n f\left( {t,s} \right)} \right)} }
\right|} \right]
\end{align*}
\begin{align*}
&\le \frac{1}{{n!\left( {b - a} \right)\left( {d - c}
\right)}}\left[ {\left( {b - x} \right)\left( {d - y}
\right)\mathop {\max }\limits_{\scriptstyle t \in \left[ {a,x}
\right] \hfill \atop \scriptstyle s \in \left[ {c,y} \right]
\hfill} \left\{ {\left( {x - t} \right)^n \left( {y - s} \right)^n
} \right\} \cdot \bigvee_c^y\bigvee_a^x\left( {D^n f} \right)}
\right.
\\
&\qquad + \left( {b - x} \right)\left( {y - c} \right)\mathop
{\max }\limits_{\scriptstyle t \in \left[ {a,x} \right] \hfill
\atop \scriptstyle s \in \left[ {y,d} \right] \hfill} \left\{
{\left( {x - t} \right)^n \left( {s - y} \right)^n } \right\}
\cdot \bigvee_y^d\bigvee_a^x\left( {D^n f} \right)
\\
&\qquad+\left( {x - a} \right)\left( {d - y} \right)\mathop {\max
}\limits_{\scriptstyle t \in \left[ {x,b} \right] \hfill \atop
\scriptstyle s \in \left[ {c,y} \right] \hfill} \left\{ {\left( {t
- x} \right)^n \left( {y - s} \right)^n } \right\} \cdot
\bigvee_c^y\bigvee_x^b\left( {D^n f} \right)
\\
&\qquad\left. {+\left( {x - a} \right)\left( {y - c}
\right)\mathop {\max }\limits_{\scriptstyle t \in \left[ {x,b}
\right] \hfill \atop \scriptstyle s \in \left[ {y,d} \right]
\hfill} \left\{ {\left( {t - x} \right)^n \left( {s - y} \right)^n
} \right\} \cdot \bigvee_y^d\bigvee_x^b\left( {D^n f} \right)}
\right]
\\
&\le \frac{1}{{n!\left( {b - a} \right)\left( {d - c}
\right)}}\left[ { \left( {x - a} \right)^n \left( {b - x}
\right)\left( {y - c} \right)^n \left( {d - y} \right) \cdot
\bigvee_c^y\bigvee_a^x\left( {D^n f} \right)} \right.
\\
&\qquad + \left( {x - a} \right)^n \left( {b - x} \right)\left( {d
- y} \right)^n \left( {y - c} \right) \cdot
\bigvee_y^d\bigvee_a^x\left( {D^n f} \right)
\\
&\qquad+\left( {b - x} \right)^n \left( {x - a} \right)\left( {y -
c} \right)^n \left( {d - y} \right) \cdot
\bigvee_c^y\bigvee_x^b\left( {D^n f} \right)
\\
&\qquad\left. {+\left( {b - x} \right)^n \left( {x - a}
\right)\left( {d - y} \right)^n \left( {y - c} \right) \cdot
\bigvee_y^d\bigvee_x^b\left( {D^n f} \right)} \right]
\\
&\le \frac{1}{{n!\left( {b - a} \right)\left( {d - c}
\right)}}\left[ { \left( {x - a} \right)^n \left( {b - x} \right)
\max \left\{ {\left( {y - c} \right)^n \left( {d - y}
\right),\left( {y - c} \right)\left( {d - y} \right)^n } \right\}
 \cdot
\bigvee_c^d\bigvee_a^x\left( {D^n f} \right)} \right.
\\
&\qquad\left.{+\left( {x - a} \right) \left( {b - x} \right)^n
\max \left\{ {\left( {y - c} \right)^n \left( {d - y}
\right),\left( {y - c} \right)\left( {d - y} \right)^n } \right\}
 \cdot
\bigvee_c^d\bigvee_x^b\left( {D^n f} \right)}\right]
\\
&\le \frac{1}{{n!\left( {b - a} \right)\left( {d - c} \right)}}
 \cdot \max \left\{
{\left( {y - c} \right)^n \left( {d - y} \right),\left( {y - c}
\right)\left( {d - y} \right)^n } \right\}
\\
&\qquad \times\max \left\{ {\left( {x - a} \right)^n \left( {b -
x} \right),\left( {x - a} \right)\left( {b - x} \right)^n }
\right\}
 \cdot
\bigvee_c^d\bigvee_a^b\left( {D^n f} \right)
\\
&\le   \frac{{\left( {b - a} \right)^{n }\left( {d - c} \right)^{n
} }}{{2^{2n + 2} (n)! }}  \cdot \bigvee_c^d\bigvee_a^b\left( {D^n
f} \right),
\end{align*}
and the first inequality in (\ref{eq5.6.17}) is proved.
\end{proof}
\newpage
\begin{remark}
Under the assumptions of Theorem \ref{thm5.6.4}, for all $(x,y)\in
Q$, the case $n = 0$ provides the following inequality:
\begin{align*}
\left| {\mathcal{B}_0\left( {f,Q} \right)} \right| &\le   \left[
{\frac{1}{2} + \frac{{\left| {x - \frac{{a + b}}{2}} \right|}}{{b
- a}}} \right]\left[ {\frac{1}{2} + \frac{{\left| {y - \frac{{c +
d}}{2}} \right|}}{{d - c}}} \right]
 \cdot
\bigvee_c^d\bigvee_a^b\left( { f} \right)
\nonumber\\
&\le    \frac{1}{{4 }}  \cdot \bigvee_c^d\bigvee_a^b\left( {f}
\right).
\end{align*}
\end{remark}

\noindent $\bullet$ \textbf{Midpoint formula:} Now, if we denote
\begin{multline}
\mathcal{E}^n_M \left( {f;Q} \right) =\frac{f(a,c)+ f(a,d)
+f(b,c)+ f(b,d)}{4}
\\
+ \frac{1}{2^{n+2}}\sum\limits_{j = 1}^n \frac{1}{{j!}}\left(
{\begin{array}{*{20}c}
   n  \\
   j  \\
\end{array}} \right)
\left( {b - a } \right)^{n - j} \left( {d - c} \right)^{ j}\left\{
{ D^{n}f\left( {a,c } \right) + (-1)^j D^{n}f\left( {a,d }
\right)} \right.
\\
+ \left. { (-1)^j D^{n}f\left( {b,c } \right)+ D^{n}f\left( {b,d }
\right) } \right\}\label{eq3.19}
\end{multline}
and
\begin{align}
\label{eq3.20}\mathcal{F}^n_M \left( {f;Q} \right)
=\frac{1}{{\left( {b - a} \right)\left( {d - c} \right)}}\int_c^d
{\int_a^b {M_n \left( {t,s} \right)d_t d_s \left( {D^n f\left(
{t,s} \right)} \right)}}
\end{align}
where,
\begin{multline*}
M_n \left( {t,s} \right) = \frac{\left( {b - a} \right)\left( {d -
c} \right)}{{4n!}}
\\
\times\left\{
\begin{array}{l}
 \left( {\frac{a+b}{2} - t} \right)^n \left( {\frac{c+d}{2} - s} \right)^n ,\,\,\,\,\,\,\,\,\,\,\,\,\,\,\,\,\,\,\,\,\,\,\,\,\,\,\,\,\,\,\,\,a \le t \le \frac{a+b}{2},\,\,\,\,\,\,\,\,c \le s \le \frac{c+d}{2} \\
 \left( { - 1} \right)^{n+1} \left( {t - \frac{a+b}{2}} \right)^n \left( {\frac{c+d}{2} - s} \right)^n,\,\,\,\,\,\frac{a+b}{2} < t \le b,\,\,\,\,\,\,\,\,\,\,\,\,\,\,c \le s \le \frac{c+d}{2} \\
 \left( { - 1} \right)^{n+1} \left( {\frac{a+b}{2} - t} \right)^n \left( {s - \frac{c+d}{2}} \right)^n,\,\,\,\,\,\,\,\,\,\,a \le t \le \frac{a+b}{2},\,\,\frac{c+d}{2} < s \le d \\
 \left( {t - \frac{a+b}{2}} \right)^n\left( {s - \frac{c+d}{2}} \right)^n,\,\,\,\,\,\,\,\,\,\,\,\,\,\,\,\,\,\,\,\,\,\,\,\,\,\,\frac{a+b}{2} < t \le b,\,\,\,\,\,\,\,\,\frac{c+d}{2} < s \le d \\
 \end{array} \right.
\end{multline*}
then we can approximate the value of a function at its midpoint in
terms of values of the function and its partial derivatives taken
at the end points with  error $\mathcal{F}_M(f;Q)$. Namely, we
have the representation formula
\begin{align}
f\left( {\frac{{a + b}}{2},\frac{{c + d}}{2}} \right) =
\mathcal{E}^n_M \left( {f;Q} \right) + \mathcal{F}^n_M \left(
{f;Q} \right)\label{eq3.21}
\end{align}
The absolute value of the error can be bounded as follows:
\begin{corollary}
With the assumptions of Theorem {\ref{thm5.6.4}} for $f$, $Q$ and
$n$, we have the inequality
\begin{align}
\left| { \mathcal{F}^n_M \left( {f;Q} \right) } \right| \le
\frac{{\left( {b - a} \right)^n \left( {d - c} \right)^n }}{{2^{2n
+ 2} n!}} \cdot \bigvee_c^d\bigvee_a^b\left( {f} \right)
\end{align}
\end{corollary}

For another assumptions on $f$, the case when $D^n f$ are
absolutely continuous on $Q$ for all ($n\ge0$), we give the
following estimates for the remainder $\mathcal{B}_n\left( {f,Q}
\right)$:
\begin{theorem}
\label{thm5.6.11}Let $f:Q\to \mathbb{R}$ be a real valued function
which has continuous $(n-1)$ partial derivatives ($n\ge 1$). If
the $n$-th partial derivatives $D^{n}f$ ($n\ge1$) are absolutely
continuous on $Q$, then for any $(x,y) \in Q$ we have
\begin{align}
\label{eq3.28}\left| {\mathcal{B}_n\left( {f,Q} \right)} \right|
 &\le\left\{
\begin{array}{l}
 \frac{1}{{\left( {n} \right)!\left( {n + 1} \right)^2\left( {b -
a} \right)\left( {d - c} \right)}} \left[{\left( {b - x}
\right)\left( {x - a} \right)^{n + 1}+ \left( {x - a}
\right)\left( {b - x} \right)^{n + 1}}\right]
\\
  \times \left[{\left( {d - y} \right) \left( {y - c}
\right)^{n + 1} + \left( {y - c} \right) \left( {d - y} \right)^{n
+ 1} }\right]  \cdot \left\| {D^{n + 1} f} \right\|_{Q,\infty },
\,\,\,\,\,{\rm{if}}\,\,\,D^{n + 1} f \in L_\infty(Q) ;\\
  \\
 \frac{1}{{\left( {n} \right)!\left( {nq + 1} \right)^{1/q}\left(
{b - a} \right)\left( {d - c} \right)}} \left[{\left( {b - x}
\right)\left( {x - a} \right)^{n + 1/q} + \left( {x - a}
\right)\left( {b - x} \right)^{n + 1/q} }\right]
\\
 \times  \left[{\left( {d - y} \right) \left( {y - c}
\right)^{n + 1/q} + \left( {y - c} \right) \left( {d - y}
\right)^{n + 1/q}  }\right] \cdot \left\| {D^{n + 1} f}
\right\|_{Q,p }, \,\,\,\,{\rm{if}}\,\,\,D^{n
+ 1} f \in L_p(Q),\\
\,\,\,\,\,\,\,\,\,\,\,\,\,\,\,\,\,\,\,\,\,\,\,\,\,\,\,\,\,\,\,\,\,\,\,\,\,\,\,\,\,\,\,\,\,\,\,
\,\,\,\,\,\,\,\,\,\,\,\,\,\,\,\,\,\,\,\,\,\,\,\,\,\,\,\,\,\,\,\,\,\,\,\,\,\,\,\,\,\,\,\,\,\,\,\,\,\,\,\,\,\,\,\,\,\,\,\,\,\,\,\,\,\,\,\,\,\,\,\,\,\,\,\,\,\,\,\,\,\,\,\,\,\,\,\,\,\,\,\,\,\,\,\,\,\,\,\,\,\,\,\,
p > 1,\frac{1}{p} + \frac{1}{q} = 1;
 \end{array} \right.
\end{align}
\end{theorem}
\begin{proof}
Since $D^n f$ is absolutely continuous on $Q$ then for any $(x,y)
\in Q$ we have the representation
\begin{multline}
\label{eq5.6.29}f\left( {x,y} \right) = \frac{1}{(b - a)(d -
c)}\left[ {(b - x)(d - y)f(a,c) + (b - x)(y - c) f(a,d)} \right.
\\
\left. {+(x - a)(d - y) f(b,c) + (x - a)(y - c)f(b,d) } \right] +
\frac{(y-c)(d - y)}{(b - a)(d - c)}
\\
\times\sum\limits_{j = 1}^n \frac{1}{{j!}}\left(
{\begin{array}{*{20}c}
   n  \\
   j  \\
\end{array}} \right)
\left\{ {(b - x)\left( {x - a } \right)^{n - j} \left[ { \left( {y
- c } \right)^{j-1} D^{n}f\left( {a,c } \right)} \right.} \right.
+ \left. {(-1)^j\left( {d - y } \right)^{j-1} D^{n}f\left( {a,d }
\right)} \right]
\\
+ \left. {(x - a)\left( {b - x } \right)^{n - j} \left[ {
(-1)^j\left( {y - c } \right)^{j-1} D^{n}f\left( {b,c } \right)+
\left( {d - y } \right)^{j-1} D^{n}f\left( {b,d } \right)} \right]
} \right\}
\\
+ \frac{1}{{\left( {b - a} \right)\left( {d - c} \right)}}\int_c^d
{\int_a^b {S_n \left( {x,t;y,s} \right) D^{n+1}f\left( {t,s}
\right) dt}ds},
\end{multline}
where the integral is considered in the Lebesgue sense and the
kernel $S_n \left( {x,t;y,s} \right)$ is given in Theorem
\ref{thm5.6.1}. Utilizing the properties of the Stieltjes
integral, we have
\begin{align*}
\left| {\mathcal{B}_n\left( {f;Q} \right)} \right| &=
\frac{1}{{\left( {b - a} \right)\left( {d - c} \right)}}\left|
{\int_c^d {\int_a^b {S_n \left( {x,t;y,s} \right) D^{n+1}f\left(
{t,s} \right) dt}ds}} \right|
\\
&=\left| {\frac{1}{{n!\left( {b - a} \right)\left( {d - c}
\right)}}\left[ {\int_c^y {\int_a^x {\left( {x - t} \right)^n
\left( {b - x} \right)\left( {y - s} \right)^n \left( {d - y}
\right) D^{n+1} f\left( {t,s} \right)dt} ds} } \right.} \right.
\\
&\qquad+ \int_y^d {\int_a^x {\left( { - 1} \right)^{n + 1} \left(
{x - t} \right)^n \left( {b - x} \right)\left( {s - y} \right)^n
\left( {y - c} \right) D^{n+1} f\left( {t,s} \right)dt}ds}
\\
&\qquad+ \int_c^y {\int_x^b {\left( { - 1} \right)^{n + 1} \left(
{t - x} \right)^n \left( {x - a} \right)\left( {y - s} \right)^n
\left( {d - y} \right)D^{n+1} f\left( {t,s} \right)dt}ds }
\\
&\qquad\left. {\left. { + \int_y^d {\int_x^b {\left( {t - x}
\right)^n \left( {x - a} \right)\left( {y - s} \right)^n \left( {d
- y} \right)D^{n+1} f\left( {t,s} \right)dt} ds} } \right]}
\right|
\end{align*}
\begin{align*}
&\le \frac{1}{{n!\left( {b - a} \right)\left( {d - c}
\right)}}\left[ {\left| {\int_c^y {\int_a^x {\left( {x - t}
\right)^n \left( {b - x} \right)\left( {y - s} \right)^n \left( {d
- y} \right) D^{n+1} f\left( {t,s} \right)dt} ds}} \right|}
\right.
\\
&\qquad+ \left| {\int_y^d {\int_a^x {\left( { - 1} \right)^{n + 1}
\left( {x - t} \right)^n \left( {b - x} \right)\left( {s - y}
\right)^n \left( {y - c} \right) D^{n+1} f\left( {t,s}
\right)dt}ds} } \right|
\\
&\qquad+ \left| {\int_c^y {\int_x^b {\left( { - 1} \right)^{n + 1}
\left( {t - x} \right)^n \left( {x - a} \right)\left( {y - s}
\right)^n \left( {d - y} \right)D^{n+1} f\left( {t,s} \right)dt}ds
} } \right|
\\
&\qquad\left. { + \left| {\int_y^d {\int_x^b {\left( {t - x}
\right)^n \left( {x - a} \right)\left( {y - s} \right)^n \left( {d
- y} \right)D^{n+1} f\left( {t,s} \right)dt} ds} } \right|}
\right]
\\
&\le \frac{1}{{n!\left( {b - a} \right)\left( {d - c}
\right)}}\left[ {\left( {b - x} \right)\left( {d - y}
\right)\int_c^y {\int_a^x {\left( {x - t} \right)^n \left( {y - s}
\right)^n \left| {D^{n+1} f\left( {t,s} \right)} \right|dt} ds}}
\right.
\\
&\qquad+ \left( {y - c} \right)\left( {b - x} \right)\int_y^d
{\int_a^x { \left( {x - t} \right)^n \left( {s - y} \right)^n
\left| {D^{n+1} f\left( {t,s} \right)} \right|dt}ds}
\\
&\qquad+ \left( {d - y}\right) \left( {x - a} \right)\int_c^y
{\int_x^b {\left( {t - x} \right)^n \left( {y - s} \right)^n
\left| {D^{n+1} f\left( {t,s} \right) } \right|dt}ds }
\\
&\qquad\left. { + \left( {d - y} \right)\left( {x - a}
\right)\int_y^d {\int_x^b {\left( {t - x} \right)^n \left( {y - s}
\right)^n \left| {D^{n+1} f\left( {t,s} \right)} \right|dt} ds} }
\right].
\end{align*}
Utilizing properties of integration together with H\"{o}lder
integral inequality for the Lebesgue integral we have
\begin{multline}
\label{eq5.6.30}\int_c^y {\int_a^x {\left( {x - t} \right)^n
\left( {y - s} \right)^n \left| {D^{n+1} f\left( {t,s} \right)}
\right|dt} ds}
\\
= \left\{ \begin{array}{l}
 \frac{{\left( {x - a} \right)^{n + 1} \left( {y - c} \right)^{n + 1} }}{{\left( {n + 1} \right)^2 }} \cdot \left\| {D^{n + 1} f} \right\|_{\left[ {a,x} \right] \times \left[ {c,y} \right],\infty } ,\,\,\,\,\,\,\,\,\,\,\,\,\,\,D^{n + 1} f \in L_\infty  \left[ {a,x} \right] \times \left[ {c,y} \right]; \\
  \\
 \frac{{\left( {x - a} \right)^{n + 1/q} \left( {y - c} \right)^{n + 1/q} }}{{\left( {nq + 1} \right)^{1/q} }} \cdot \left\| {D^{n + 1} f} \right\|_{\left[ {a,x} \right] \times \left[ {c,y} \right],p} ,\,\,\,\,\,\,\,\,\,\,D^{n + 1} f \in L_p \left[ {a,x} \right] \times \left[ {c,y} \right],\\
 \,\,\,\,\,\,\,\,\,\,\,\,\,\,\,\,\,\,\,\,\,\,\,\,\,\,\,\,\,\,\,\,\,\,\,\,\,\,\,\,\,\,\,\,\,\,\,\,\,\,\,\,\,\,\,\,\,\,\,\,\,\,\,\,\,\,\,\,\,\,\,\,\,\,\,\,\,\,\,\,\,\,\,\,\,\,\,\,\,\,\,\,\,\,\,\,\,\,\,\,\,\,\,\,\,\,\,\,\,\,\,\,\,\,\,\,\,\,\,\,\,\,\,\,\,\,\,\,p > 1,\frac{1}{p} + \frac{1}{q} = 1;
 \end{array} \right.,
\end{multline}
\begin{multline}
\label{eq5.6.31}\int_y^d {\int_a^x { \left( {x - t} \right)^n
\left( {s - y} \right)^n \left| {D^{n+1} f\left( {t,s} \right)}
\right|dt}ds}
\\
= \left\{ \begin{array}{l}
 \frac{{\left( {x - a} \right)^{n + 1} \left( {d - y} \right)^{n + 1} }}{{\left( {n + 1} \right)^2 }} \cdot \left\| {D^{n + 1} f} \right\|_{\left[ {a,x} \right] \times \left[ {y,d} \right],\infty } ,\,\,\,\,\,\,\,\,\,\,\,\,\,\,D^{n + 1} f \in L_\infty  \left[ {a,x} \right] \times \left[ {y,d} \right]; \\
  \\
 \frac{{\left( {x - a} \right)^{n + 1/q} \left( {d - y} \right)^{n + 1/q} }}{{\left( {nq + 1} \right)^{1/q} }} \cdot \left\| {D^{n + 1} f} \right\|_{\left[ {a,x} \right] \times \left[ {y,d} \right],p} ,\,\,\,\,\,\,\,\,\,\,D^{n + 1} f \in L_p \left[ {a,x} \right] \times \left[ {y,d} \right],\\
 \,\,\,\,\,\,\,\,\,\,\,\,\,\,\,\,\,\,\,\,\,\,\,\,\,\,\,\,\,\,\,\,\,\,\,\,\,\,\,\,\,\,\,\,\,\,\,\,\,\,\,\,\,\,\,\,\,\,\,\,\,\,\,\,\,\,\,\,\,\,\,\,\,\,\,\,\,\,\,\,\,\,\,\,\,\,\,\,\,\,\,\,\,\,\,\,\,\,\,\,\,\,\,\,\,\,\,\,\,\,\,\,\,\,\,\,\,\,\,\,\,\,\,\,\,\,\,\,p > 1,\frac{1}{p} + \frac{1}{q} = 1;
  \end{array} \right.,
\end{multline}
\begin{multline}
\label{eq5.6.32}\int_c^y {\int_x^b {\left( {t - x} \right)^n
\left( {y - s} \right)^n \left| {D^{n+1} f\left( {t,s} \right) }
\right|dt}ds }
\\
= \left\{ \begin{array}{l}
 \frac{{\left( {b - x} \right)^{n + 1} \left( {y - c} \right)^{n + 1} }}{{\left( {n + 1} \right)^2 }} \cdot \left\| {D^{n + 1} f} \right\|_{\left[ {x,b} \right] \times \left[ {c,y} \right],\infty } ,\,\,\,\,\,\,\,\,\,\,\,\,\,\,D^{n + 1} f \in L_\infty  \left[ {x,b} \right] \times \left[ {c,y} \right]; \\
  \\
 \frac{{\left( {b - x} \right)^{n + 1/q} \left( {y - c} \right)^{n + 1/q} }}{{\left( {nq + 1} \right)^{1/q} }} \cdot \left\| {D^{n + 1} f} \right\|_{\left[ {x,b} \right] \times \left[ {c,y} \right],p} ,\,\,\,\,\,\,\,\,\,\,D^{n + 1} f \in L_p \left[ {x,b} \right] \times \left[ {c,y} \right],\\
 \,\,\,\,\,\,\,\,\,\,\,\,\,\,\,\,\,\,\,\,\,\,\,\,\,\,\,\,\,\,\,\,\,\,\,\,\,\,\,\,\,\,\,\,\,\,\,\,\,\,\,\,\,\,\,\,\,\,\,\,\,\,\,\,\,\,\,\,\,\,\,\,\,\,\,\,\,\,\,\,\,\,\,\,\,\,\,\,\,\,\,\,\,\,\,\,\,\,\,\,\,\,\,\,\,\,\,\,\,\,\,\,\,\,\,\,\,\,\,\,\,\,\,\,\,\,\,\,p > 1,\frac{1}{p} + \frac{1}{q} = 1;
  \end{array} \right.,
\end{multline}
and
\begin{multline}
\label{eq5.6.33}\int_y^d {\int_x^b {\left( {t - x} \right)^n
\left( {y - s} \right)^n \left| {D^{n+1} f\left( {t,s} \right)}
\right|dt} ds}
\\
= \left\{ \begin{array}{l}
 \frac{{\left( {b - x} \right)^{n + 1} \left( {d - y} \right)^{n + 1} }}{{\left( {n + 1} \right)^2 }} \cdot \left\| {D^{n + 1} f} \right\|_{\left[ {x,b} \right] \times \left[ {y,d} \right],\infty } ,\,\,\,\,\,\,\,\,\,\,\,\,\,\,D^{n + 1} f \in L_\infty  \left[ {x,b} \right] \times \left[ {y,d} \right]; \\
  \\
 \frac{{\left( {b - x} \right)^{n + 1/q} \left( {d - y} \right)^{n + 1/q} }}{{\left( {nq + 1} \right)^{1/q} }} \cdot \left\| {D^{n + 1} f} \right\|_{\left[ {x,b} \right] \times \left[ {y,d} \right],p} ,\,\,\,\,\,\,\,\,\,\,D^{n + 1} f \in L_p \left[ {x,b} \right] \times \left[ {y,d} \right],\\
 \,\,\,\,\,\,\,\,\,\,\,\,\,\,\,\,\,\,\,\,\,\,\,\,\,\,\,\,\,\,\,\,\,\,\,\,\,\,\,\,\,\,\,\,\,\,\,\,\,\,\,\,\,\,\,\,\,\,\,\,\,\,\,\,\,\,\,\,\,\,\,\,\,\,\,\,\,\,\,\,\,\,\,\,\,\,\,\,\,\,\,\,\,\,\,\,\,\,\,\,\,\,\,\,\,\,\,\,\,\,\,\,\,\,\,\,\,\,\,\,\,\,\,\,\,\,\,\,p > 1,\frac{1}{p} + \frac{1}{q} = 1;
  \end{array} \right.
\end{multline}
Adding (\ref{eq5.6.30})--(\ref{eq5.6.33}) we deduce that
\begin{multline*}
 \left| {\mathcal{B}_n\left( {f,Q} \right)} \right|
\\
  \le \frac{\left( {b - x} \right)\left( {d - y} \right)}{{n!\left(
{b - a} \right)\left( {d - c} \right)}} \times\left\{
\begin{array}{l}
 \frac{{\left( {x - a} \right)^{n + 1} \left( {y - c} \right)^{n + 1} }}{{\left( {n + 1} \right)^2 }} \cdot \left\| {D^{n + 1} f} \right\|_{\left[ {a,x} \right] \times \left[ {c,y} \right],\infty } ,\,\,D^{n + 1} f \in L_\infty  \left[ {a,x} \right] \times \left[ {c,y} \right] ; \\
  \\
 \frac{{\left( {x - a} \right)^{n + 1/q} \left( {y - c} \right)^{n + 1/q} }}{{\left( {nq + 1} \right)^{1/q} }} \cdot \left\| {D^{n + 1} f} \right\|_{\left[ {a,x} \right] \times \left[ {c,y} \right],p} ,\,\,D^{n + 1} f \in L_p  \left[ {a,x} \right] \times \left[ {c,y} \right],\\
 \,\,\,\,\,\,\,\,\,\,\,\,\,\,\,\,\,\,\,\,\,\,\,\,\,\,\,\,\,\,\,\,\,\,\,\,\,\,\,\,\,\,\,\,\,\,\,\,\,\,\,\,\,\,\,\,\,\,\,\,\,\,\,\,\,\,\,\,\,\,\,\,\,\,\,\,\,\,\,\,\,\,\,\,\,\,\,\,\,\,\,\,\,\,\,\,\,\,\,\,\,\,\,\,p > 1,\frac{1}{p} + \frac{1}{q} = 1;
  \end{array} \right.\nonumber\\
\\
 + \frac{\left( {b - x} \right)\left( {y - c}
\right)}{{n!\left( {b - a} \right)\left( {d - c} \right)}}
\times\left\{
\begin{array}{l}
 \frac{{\left( {x - a} \right)^{n + 1} \left( {d - y} \right)^{n + 1} }}{{\left( {n + 1} \right)^2 }} \cdot \left\| {D^{n + 1} f} \right\|_{\left[ {a,x} \right] \times \left[ {y,d} \right],\infty } ,\,\,D^{n + 1} f \in L_\infty  \left[ {a,x} \right] \times \left[ {y,d} \right] ; \\
  \\
 \frac{{\left( {x - a} \right)^{n + 1/q} \left( {d - y} \right)^{n + 1/q} }}{{\left( {nq + 1} \right)^{1/q} }} \cdot \left\| {D^{n + 1} f} \right\|_{\left[ {a,x} \right] \times \left[ {y,d} \right],p} ,\,\,D^{n + 1} f \in L_p  \left[ {a,x} \right] \times \left[ {y,d} \right],\\
 \,\,\,\,\,\,\,\,\,\,\,\,\,\,\,\,\,\,\,\,\,\,\,\,\,\,\,\,\,\,\,\,\,\,\,\,\,\,\,\,\,\,\,\,\,\,\,\,\,\,\,\,\,\,\,\,\,\,\,\,\,\,\,\,\,\,\,\,\,\,\,\,\,\,\,\,\,\,\,\,\,\,\,\,\,\,\,\,\,\,\,\,\,\,\,\,\,\,\,\,\,\,\,\,p > 1,\frac{1}{p} + \frac{1}{q} = 1;
 \end{array} \right.
\nonumber\\\\
 + \frac{\left( {x - a} \right)\left( {d - y}
\right)}{{n!\left( {b - a} \right)\left( {d - c} \right)}}
\times\left\{
\begin{array}{l}
 \frac{{\left( {b - x} \right)^{n + 1} \left( {y - c} \right)^{n + 1} }}{{\left( {n + 1} \right)^2 }} \cdot \left\| {D^{n + 1} f} \right\|_{\left[ {x,b} \right] \times \left[ {c,y} \right],\infty } ,\,\,D^{n + 1} f \in L_\infty  \left[ {x,b} \right] \times \left[ {c,y} \right] ; \\
  \\
 \frac{{\left( {b - x} \right)^{n + 1/q} \left( {y - c} \right)^{n + 1/q} }}{{\left( {nq + 1} \right)^{1/q} }} \cdot \left\| {D^{n + 1} f} \right\|_{\left[ {x,b} \right] \times \left[ {c,y} \right],p} ,\,\,D^{n + 1} f \in L_p  \left[ {x,b} \right] \times \left[ {c,y} \right],\\
 \,\,\,\,\,\,\,\,\,\,\,\,\,\,\,\,\,\,\,\,\,\,\,\,\,\,\,\,\,\,\,\,\,\,\,\,\,\,\,\,\,\,\,\,\,\,\,\,\,\,\,\,\,\,\,\,\,\,\,\,\,\,\,\,\,\,\,\,\,\,\,\,\,\,\,\,\,\,\,\,\,\,\,\,\,\,\,\,\,\,\,\,\,\,\,\,\,\,\,\,\,\,\,\,p > 1,\frac{1}{p} + \frac{1}{q} = 1;
 \end{array} \right.
\end{multline*}
\begin{align*}
&\qquad+ \frac{\left( {x - a} \right)\left( {y - c}
\right)}{{n!\left( {b - a} \right)\left( {d - c} \right)}}
\times\left\{
\begin{array}{l}
 \frac{{\left( {b - x} \right)^{n + 1} \left( {d - y} \right)^{n + 1} }}{{\left( {n + 1} \right)^2 }} \cdot \left\| {D^{n + 1} f} \right\|_{\left[ {x,b} \right] \times \left[ {y,d} \right],\infty } ,\,\,D^{n + 1} f \in L_\infty  \left[ {x,b} \right] \times \left[ {y,d} \right] ; \\
  \\
 \frac{{\left( {b - x} \right)^{n + 1/q} \left( {d - y} \right)^{n + 1/q} }}{{\left( {nq + 1} \right)^{1/q} }} \cdot \left\| {D^{n + 1} f} \right\|_{\left[ {x,b} \right] \times \left[ {y,d} \right],p} ,\,\,D^{n + 1} f \in L_p  \left[ {x,b} \right] \times \left[ {y,d} \right],\\
 \,\,\,\,\,\,\,\,\,\,\,\,\,\,\,\,\,\,\,\,\,\,\,\,\,\,\,\,\,\,\,\,\,\,\,\,\,\,\,\,\,\,\,\,\,\,\,\,\,\,\,\,\,\,\,\,\,\,\,\,\,\,\,\,\,\,\,\,\,\,\,\,\,\,\,\,\,\,\,\,\,\,\,\,\,\,\,\,\,\,\,\,\,\,\,\,\,\,\,\,\,\,\,\,p > 1,\frac{1}{p} + \frac{1}{q} = 1;
 \end{array} \right.
\\
&\le \left\{
\begin{array}{l}
 \frac{1}{{\left( {n} \right)!\left( {n + 1} \right)^2\left( {b -
a} \right)\left( {d - c} \right)}} \left[{\left( {b - x}
\right)\left( {x - a} \right)^{n + 1}+ \left( {x - a}
\right)\left( {b - x} \right)^{n + 1}}\right]
\\
  \times \left[{\left( {d - y} \right) \left( {y - c}
\right)^{n + 1} + \left( {y - c} \right) \left( {d - y} \right)^{n
+ 1} }\right]  \cdot \left\| {D^{n + 1} f} \right\|_{Q,\infty },\\
\,\,\,\,\,\,\,\,\,\,\,\,\,\,\,\,\,\,\,\,\,\,\,\,\,\, \,\,\,\,\,\,\,\,\,\,\,\,\,\,\,\,\,\,\,\,\,\,\,\,\,\,\,\,\,\,\,\,\,\,\,\,\,\,\,\,\,\,\,\,\,\,\,\,\,\,\,\,\,\,\,\,\,\,\,\,\,\,\,\,\,\,\,\,\,\,\,\,\,\,\,\,\,\,\,\,\,\,\,\,\,\,\,\,\,\,\,\,\,\,\,\,\,\,\,\,\,\,\,\,\,\,\,\,\,\,\,\,\,\,\,\,\,\,\,\,\,\,{\rm{if}}\,\,\,D^{n + 1} f \in L_\infty(Q_{x,c}^{b,y}) ; \\
  \\
 \frac{1}{{\left( {n} \right)!\left( {nq + 1} \right)^{1/q}\left(
{b - a} \right)\left( {d - c} \right)}} \left[{\left( {b - x}
\right)\left( {x - a} \right)^{n + 1/q} + \left( {x - a}
\right)\left( {b - x} \right)^{n + 1/q} }\right]
\\
 \times  \left[{\left( {d - y} \right) \left( {y - c}
\right)^{n + 1/q} + \left( {y - c} \right) \left( {d - y}
\right)^{n + 1/q}  }\right] \cdot \left\| {D^{n + 1} f}
\right\|_{Q,p },
\\
\,\,\,\,\,\,\,\,\,\,\,\,\,\,\,\,\,\,\,\,\,\,\,\,\,\,
\,\,\,\,\,\,\,\,\,\,\,\,\,\,\,\,\,\,\,\,\,\,\,\,\,\,\,\,\,\,\,\,\,\,\,\,\,\,\,\,\,\,\,\,\,\,\,\,\,\,\,\,\,\,\,\,\,\,\,\,\,\,\,\,\,\,\,\,\,\,\,\,\,\,\,\,\,\,\,\,\,\,\,\,\,\,\,\,\,\,\,\,\,\,\,\,\,\,\,\,\,\,\,\,\,\,\,\,\,\,\,\,\,\,\,\,\,\,\,\,\,\,\,{\rm{if}}\,\,\,D^{n
+ 1} f \in L_p(Q_{x,c}^{b,y}),\\
\,\,\,\,\,\,\,\,\,\,\,\,\,\,\,\,\,\,\,\,\,\,\,\,\,\,\,\,\,\,\,\,\,\,\,\,\,\,\,\,\,\,\,\,\,\,\,
\,\,\,\,\,\,\,\,\,\,\,\,\,\,\,\,\,\,\,\,\,\,\,\,\,\,\,\,\,\,\,\,\,\,\,\,\,\,\,\,\,\,\,\,\,\,\,\,\,\,\,\,\,\,\,\,\,\,\,\,\,\,\,\,\,\,\,\,\,\,\,\,\,\,\,\,\,\,\,\,\,\,\,\,\,\,\,\,\,\,\,\,\,\,\,\,\,\,\,\,\,\,\,\,
p > 1,\frac{1}{p} + \frac{1}{q} = 1;
 \end{array} \right.
\end{align*}
\end{proof}

Thus, the value of the function $f$ can be approximated    at its
midpoint in terms of values of $f$ and its partial derivatives
taken at the end points with  error $\mathcal{F}_M(f;Q)$, i.e.,
\begin{align*}
f\left( {\frac{{a + b}}{2},\frac{{c + d}}{2}} \right) =
\mathcal{E}_M \left( {f;Q} \right) + \mathcal{F}_M \left( {f;Q}
\right),
\end{align*}
such that the error $\mathcal{F}_M(f;Q)$ satisfies the following
bounds:
\begin{corollary}
With the assumptions of Theorem {\ref{thm5.6.11}} for $f$, $Q$ and
$n$, we have the inequality
\begin{align*}
\left| { \mathcal{F}^n_M \left( {f;Q} \right) } \right| \le
\frac{1}{{n! }} \times\left\{
\begin{array}{l}
 \frac{{\left( {b - a} \right)^{n + 1} \left( {d - c} \right)^{n + 1} }}{{2^{2n+2}\left( {n + 1} \right)^2 }} \cdot \left\| {D^{n + 1} f} \right\|_{Q,\infty } ,\,\,{\rm{if}}\,\,D^{n + 1} f \in L_\infty(Q) ; \\
  \\
 \frac{{\left( {b - a} \right)^{n + 1/q} \left( {d - c} \right)^{n + 1/q} }}{{2^{2n + 2/q}\left( {nq + 1} \right)^{1/q} }} \cdot \left\| {D^{n + 1} f} \right\|_{Q,p} ,\,\,{\rm{if}}\,\,D^{n + 1} f \in L_p(Q),\\
 \,\,\,\,\,\,\,\,\,\,\,\,\,\,\,\,\,\,\,\,\,\,\,\,\,\,\,\,\,\,\,\,\,\,\,\,\,\,\,\,\,\,\,\,\,\,\,\,\,\,\,\,\,\,\,\,\,\,\,\,\,\,\,\,\,\,\,\,\,\,\,\,\,\,\,\,\,\,\,\,\,\,\,\,\,\,\,\,\,\,\,\,\,\,p > 1,\frac{1}{p} + \frac{1}{q} = 1;
 \end{array} \right..
\end{align*}
In particular, for $n=0$, we have
\begin{align*}
\left| { \mathcal{F}^0_M \left( {f;Q} \right) } \right| \le
\left\{
\begin{array}{l}
 \frac{{\left( {b - a} \right) \left( {d - c} \right) }}{{4 }} \cdot \left\| {D f} \right\|_{Q,\infty } ,\,\,{\rm{if}}\,\,D f \in L_\infty(Q) ; \\
  \\
 \frac{{\left( {b - a} \right)^{1/q} \left( {d - c} \right)^{1/q} }}{{4^{1/q} }} \cdot \left\| {D  f} \right\|_{Q,p} ,\,\,{\rm{if}}\,\,D  f \in L_p(Q),\\
 \,\,\,\,\,\,\,\,\,\,\,\,\,\,\,\,\,\,\,\,\,\,\,\,\,\,\,\,\,\,\,\,\,\,\,\,\,\,\,\,\,\,\,\,\,\,\,\,\,\,\,\,\,\,\,\,\,\,\,\,\,\,\,\,\,\,\,\,\,p
> 1,\frac{1}{p} + \frac{1}{q} = 1;
 \end{array} \right..
\end{align*}
\end{corollary}

\begin{remark}
Different kinds of H\"{o}lder continuous functions of two variable
have been defined in literature, for example  a function $f: Q \to
\mathbb{R} $ is to be of $(\beta_1,\beta_2)$--H\"{o}lder type
mapping on the co-ordinate, if for all $\left( {t_1 ,s_1 }
\right), \left( {t_1 ,s_1 } \right) \in Q$, there exist $H_1,
H_2>0$ and $\beta_1, \beta_2>0$ such that
\begin{align*}
\left| {f\left( {t_1 ,s_1 } \right) - f\left( {t_2 ,s_2 } \right)}
\right| \le H_1 \left| {t_1  - t_2 } \right|^{\beta_1} + H_2
\left| {s_1 - s_2 } \right|^{\beta_2}.
\end{align*}
If $\beta_1 = \beta_2 = 1$, then $f$ is called $(L_1,
L_2)$--Lipschitz on the co-ordinate, i.e.,
\begin{align*}
\left| {f\left( {t_1 ,s_1 } \right) - f\left( {t_2 ,s_2 } \right)}
\right| \le L_1 \left| {t_1  - t_2 } \right| + L_2 \left| {s_1 -
s_2 } \right|.
\end{align*}
Another definition of H\"{o}lder continuous mapping can be stated
as:
\begin{align*}
\left| {f\left( {t_1 ,s_1 } \right) - f\left( {t_2 ,s_2 } \right)}
\right| \le M_{\beta _1 ,\beta _2 } \left( {t_1 ,t_2 ;s_1 ,s_2 }
\right).
\end{align*}
where, $ M_{\beta _1 ,\beta _2 } \left( {t_1 ,t_2 ;s_1 ,s_2 }
\right)$ is any real norm. In particular, if $\beta_1 = \beta_2 =
1$, the usual Euclidean norm is well known as:
\begin{align*}
M_{1,1} \left( {t_1 ,t_2 ;s_1 ,s_2 } \right): = \sqrt {\left( {t_1
- s_1 } \right)^2  + \left( {t_2  - s_2 } \right)^2 }.
\end{align*}
Under any type of H\"{o}lder continuity, one may has several
bounds for the error bounds $\mathcal{B}_n\left( {f,Q} \right)$
and $\mathcal{F}^n_M\left( {f,Q} \right)$ of the function $f$. We
leave this to the interested reader (see \cite{alomari}).
\end{remark}

\end{document}